\documentclass[12pt]{article}

\usepackage{amsmath,amssymb,amsthm,epsfig,amsfonts}
\usepackage{hyperref}
\usepackage{authblk}
\usepackage{fullpage}
\usepackage{pictex}

\newcommand{\St}{\mathop{\rm St}\nolimits}

\newcommand{\Stg}{\mathop{\rm St}\nolimits_G}

\newtheorem{theorem}{Theorem}[section]
\newtheorem{corollary}[theorem]{Corollary}
\newtheorem{proposition}[theorem]{Proposition}
\newtheorem{lemma}[theorem]{Lemma}
\newtheorem{question}[theorem]{Question}

\newtheorem*{theoremmain}{Theorem~\ref{thm:main}}

\theoremstyle{definition}
\newtheorem{definition}[theorem]{Definition}
\newtheorem{remark}[theorem]{Remark}
\newtheorem{notation}[theorem]{Notation}

\title{Explicit Generators for the Stabilizers of Rational Points in Thompson’s Group $F$}

\author[1]{Krystofer Baker}
\author[1]{Dmytro Savchuk}

\affil[1]{Department of Mathematics and Statistics\\
        University of South Florida\\
        4202 E Fowler Ave\\
        Tampa, FL 33620-5700\\
        \href{mailto:kbaker9@mail.usf.edu}{kbaker9@mail.usf.edu}, \href{mailto:savchuk@usf.edu}{savchuk@usf.edu}}

\begin{document}

\maketitle

\begin{abstract}
We construct explicit finite generating sets for the stabilizers in Thompson's group $F$ of rational points of a unit interval or the Cantor set. Our technique is based on the Reidemeister-Schreier procedure in the context of Schreier graphs of such stabilizers in $F$. It is well known that the stabilizers of dyadic rational points are isomorphic to $F\times F$ and can thus be generated by 4 explicit elements. We show that the stabilizer of every non-dyadic rational point $b\in (0,1)$ is generated by 5 elements that are explicitly calculated as words in generators $x_0, x_1$ of $F$ that depend on the binary expansion of $b$. We also provide an alternative simple proof that the stabilizers of all rational points are finitely presented.
\end{abstract}

\section{Introduction}
Research around Thompson's group $F$ has been very active since its discovery in 1965 by Richard Thompson, who first defined the group in connection with his work in logic. The group showed up in many areas since then. It led to the construction of finitely-presented groups with an unsolvable word problem~\cite{mckenzie_t:unsolvable}, served as a proposed platform group for some cryptographic schemes~\cite{shpilrain_u:thompson_crypto,matucci:crypto08}, and played an important role in homotopy theory~\cite{brown_g:FP_infty84}, combinatorics~\cite{cleary_t:combinatorial_prop_of_F04}, and, of course, group theory. Many basic (and not only) facts about $F$ are given in a classical paper by Cannon, Floyd, and Parry~\cite{cannon_fp:intro_thompson}. Belk's PhD thesis~\cite{belk:phd} can also serve as a very nicely written introduction to the area.

The group $F$ acts by orientation preserving piecewise linear homeomorphisms on the unit interval $[0,1]$ and on the Cantor set $X^\omega$ consisting of all infinite words over the alphabet $X=\{0,1\}$. These actions give rise to the stabilizer subgroups of $F$ with respect to points or sets of points in $[0,1]$ or $X^\omega$. Understanding such subgroups may give extra information about $F$. For example, one can negatively solve one of the main open questions about $F$, whether or not it is amenable (see a recent survey on this topic by Guba~\cite{guba:thompson_group_amenability_survey22}), if one constructs an infinite index subgroup $H<F$ such that the Schreier graph $\Gamma(F,H)$ of $H$ in $F$ is non-amenable. In~\cite{savchuk:thompson,savchuk:thompson2} the second author has constructed the Schreier graphs of the stabilizers of each point in $X^\omega$ and showed that the Schreier graph of the stabilizer of any finite subset of $X^\omega$ is amenable, thus confirming that these subgroups do not help to resolve the amenability question, but also offering a new geometric insight on the action of $F$ on $X^\omega$.  This construction was later used
to study random walks on the group $F$.  Kaimanovich in~\cite{kaimanovich:thompson_not_liouville} showed that $F$ does not have the Liouville property for finitely supported measures, and Juschenko and Zheng~\cite{juschenko_z:infinitely_supported18} proved that it has the Liouville property for some measures that are not finitely supported.

Schreier graphs appear naturally in many contexts. For example, papers~\cite{dangeli_dmn:schreier} and~\cite{bond_cdn:amenable} describe Schreier graphs of the actions of certain groups generated by automata on the boundary of a rooted tree. Miasnikov and the second author proved in~\cite{miasnikov_s:automatic_graph} that one of the Schreier graphs studied in~\cite{bond_cdn:amenable} is the first example of an automatic graph of intermediate growth. Schreier graphs also play an important role in the spectral computations related to groups acting on rooted trees (see, for example,~\cite{bartholdi_g:spectrum,grigorch_ss:img,grigorchuk-n:schur} and a more recent paper~\cite{grigorch_ln:subshifts_with_leading_sequences__schreier_graphs22}).

Stabilizers of points in the Cantor set also serve as the simplest examples of maximal subgroups in $F$ of infinite index, as was shown in~\cite{savchuk:thompson2}. The question on the existence of other maximal subgroups of infinite index in $F$ posed in~\cite{savchuk:thompson2} was answered by Golan and Sapir in~\cite{golan_s:on_subgroups_of_thompson_group_F17}. They used the so-called oriented subgroup of $F$ defined by Jones in relation to his study of linear representation of $F$ and $T$~\cite{jones:some_unitary_reps_of_F_and_T17}. This group was also studied by Golan and Sapir in~\cite{golan_s:jones_subgroup_of_F17}. Subsequently, other subgroups defined in Jones' program were shown to provide examples of maximal subgroups of infinite index in $F$ and other Thompson-type groups in the papers by Aiello and Nagnibeda~\cite{aiello_n:on_oriented_thompson_subgroup_F_3_22,aiello_n:3colorable_subgroup_max_sungroups_of_F23} and Golan~\cite{golan:on_maximal_subgroups_of_F}, who constructed an infinite countable family of non-isomorphic maximal subgroups of infinite index in $F$. A different source of such subgroups was discovered in Thompson's group $V$ by Belk, Bleak, Quick, and Skipper in~\cite{belk_bqs:type_systems_and_maximal_subgroups_of_V}.

The structure of the stabilizer subgroups in $F$ of finite subsets of $(0,1)$ was studied by Golan and Sapir in~\cite{golan_s:stabilizers_of_finite_sets17}, where it was shown, in particular, that the stabilizers of finite sets of rational numbers from $(0,1)$ are finitely generated. Namely, Theorem~5.9 in~\cite{golan_s:stabilizers_of_finite_sets17} claims that the stabilizer of a set consisting of $m_1$ dyadic rational numbers and $m_2$ non-dyadic rational numbers can be generated by $2m_1+m_2+2$ elements, and that this is the minimal possible size for a generating set of such subgroup. Therefore, for singleton sets consisting of rational numbers one gets 4-element generating sets for dyadic rational numbers and 3-element generating sets for non-dyadic rational numbers. However, the general construction of generating sets of the stabilizer subgroups in~\cite{golan_s:stabilizers_of_finite_sets17} is quite technical. According to Golan (private communication) in the case of a singleton set, the construction can yield a generating set in which the generators are given in the form of tree pair diagrams. The main purpose of this paper is to construct an explicit and simple generating sets for the stabilizers $\St_F(b)$ in $F$ for all rational points $b\in X^\omega$ in which the generators are given as explicit words in the generators $x_0,x_1$ of $F$ that are obtained directly from the period and preperiod of $b$. We achieve this via a completely different approach compared to~\cite{golan_s:stabilizers_of_finite_sets17}: we apply the Reidemeister-Schreier procedure to the Schreier graphs $\Gamma_b$ of the stabilizers of rational points in the Cantor set $X^\omega$, constructed in~\cite{savchuk:thompson2}. Note that since the Schreier graphs $\Gamma_b$ are infinite, the standard Reidemeister-Schreier procedure yields infinite generating sets that need to be simplified to 4- or 5-element sets. While our construction yields much simpler generating sets written explicitly as words in the generators $x_0,x_1$ of $F$, the trade-off of this approach is that it currently works only for stabilizers of singleton sets, as there is currently no good description of Schreier graphs of the stabilizers of larger finite sets of rational numbers. Our method also allows us to prove (see Proposition~\ref{prop:fp}) that the stabilizer subgroups are finitely presented, which also follows from the description of their algebraic structure from~\cite[Lemma~4.11]{golan_s:stabilizers_of_finite_sets17}.

Our main theorem is given below. To state it, we recall that the generating set $\{x_0,x_1\}$ of $F$ can be extended to an infinite generating set $\{x_0,x_1,x_2,\ldots\}$ with respect to which $F$ has the following infinite presentation
\begin{equation}
\label{eqn:infinite_presentation}
F\cong\langle x_0, x_1, x_2,\ldots\mid x_kx_nx_k^{-1}=x_{n+1} \text{ for }  0\leq k < n \rangle,
\end{equation}
so that $x_n=x_0^{n-1}x_{1}x_0^{-n-1}$ for $n\geq 2$. We also define elements $y_1=x_0^{-2}x_1x_0$ and $y_2=x_0^{-3}x_1x_0^2$ of $F$, whose graphs are obtained from the graphs of $x_1$ and $x_2$ by applying central symmetry about the center of $[0,1]^2$. In the statement below a word $w$ over any 2-letter alphabet $\{x,y\}$ will be denoted by $w(x,y)$ and, for any $g_0,g_1\in F$, $w(g_1,g_2)$ denotes the element of $F$ obtained by replacing each $x$ and $y$ in $w$ by $g_0$ and $g_1$, respectively.

\begin{theorem}
\label{thm:main}~\\[-2mm]
\begin{enumerate}
  \item[(a)] For a rational point $a=10w^\infty$ of the Cantor set $X^\omega$ with a finite word $w$ that is not a proper power, we have
  \[\St_F(a)=\langle x_2, x_3, y_1, y_2, w(x_1^{-1},x_1^{-1}x_0)\rangle.\]
  \item[(b)] Let $b\in X^\omega$ be an arbitrary rational point different from $0^\infty$ and $1^\infty$ and let $b=vw^\infty$ be its unique decomposition as an eventually periodic word with a finite nonempty period $w\in X^*$ that is not a proper power and a finite preperiod $v\in X^*$ such that the ending of $v$ differs from the one of $w$. Then there is $h\in F$ that can be explicitly computed from the Schreier graph $\Gamma_b$, such that $h(b)=10vw^\infty$ and
  \[\St_F(b)=\langle x_2^h, x_3^h, y_1^h, y_2^h, w(x_1^{-1},x_1^{-1}x_0)^h\rangle.\]
\end{enumerate}
\end{theorem}

Our result can serve as additional motivation to study Schreier graphs. Our approach can be applied to other subgroups in other groups, particularly $F$-type, where the corresponding Schreier graphs are well understood. For example, the Schreier graphs of Thompson's group $T$ with respect to the stabilizers of rational points in the Cantor set were constructed by Pennington~\cite{pennington:thesis17} and one can use the approach developed in the present paper to construct their generating sets.  In general, it would be interesting to describe Schreier graphs of other maximal subgroups of $F$ discovered by Golan, Sapir, Aiello, and Nagnibeda, and of the stabilizers of points in other groups of Thompson's family. Another interesting question is about the Schreier graphs of the minimal non-solvable group of homeomorphisms sitting inside $F$ from~\cite{bleak:minimal_non-solvable_group_of_homeomorphisms09}. Philosophically, a Schreier graph of a smaller subgroup approximates the Cayley graph of the group better, but is usually more complicated and harder to describe and work with.

The structure of the paper is as follows. We start by setting up the notation and recalling basic facts about $F$ in Section~\ref{sec:groupF}. In Section~\ref{sec:schreier} we recall the construction of Schreier graphs for $F$ and introduce some more notation. Finally, Section~\ref{sec:generators} contains the proof of the main result.

\noindent \textbf{Acknowledgements.} The authors would like to thank Gili Golan for her comments on the historical background and to the anonymous referee for careful reading of the paper and providing several suggestions that improved the exposition.

\section{Thompson's group $F$}
\label{sec:groupF}

We start from introducing notation and recalling some basic properties $F$ that we will use later. Most basic properties of $F$ are nicely surveyed in~\cite{cannon_fp:intro_thompson}.

\begin{definition}
Thompson's group $F$ is the group of all strictly increasing piecewise linear homeomorphisms from the closed unit interval $[0,1]$ to itself that are differentiable everywhere except at finitely many dyadic rational numbers and such that on the intervals of differentiability the derivatives are integer powers of 2. The group operation is composition of homeomorphisms.
\end{definition}

The group $F$ is generated by two elements $x_0, x_1$ defined as follows.

\begin{equation}
\label{eqn:x_0x_1}
x_0(t)=\left\{
\begin{array}{ll}
\frac t2,&0\leq t\leq\frac12,\\[1mm]
t-\frac14,&\frac12\leq t\leq\frac34,\\[1mm]
2t-1,&\frac34\leq t\leq1,\\
\end{array} \right.
\qquad x_1(t)=\left\{
\begin{array}{ll}
t,&0\leq t\leq\frac12,\\[1mm]
\frac t2+\frac14,&\frac12\leq t\leq\frac34,\\[1mm]
t-\frac18,&\frac34\leq t\leq\frac78,\\[1mm]
2t-1,&\frac78\leq t\leq1.\\
\end{array} \right.
\end{equation}

The graphs of $x_0$ and $x_1$ are depicted in
Figure~\ref{fig_gens}.
\begin{figure}[h]
\begin{center}
\includegraphics{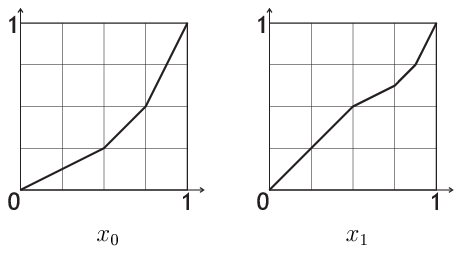}
\end{center}
\caption{Generators $x_0$ and $x_1$ of $F$\label{fig_gens}}
\end{figure}

Throughout the paper we will adopt the convention to use the right actions as was done in~\cite{savchuk:thompson2}.

\begin{notation}
For any two elements $f$, $g$ of $F$ and any $t\in[0,1]$
\begin{equation}
\label{eqn_conv}
(fg)(t)=g(f(t)),\quad f^g=gfg^{-1}, \quad [f,g]=fgf^{-1}g^{-1}.
\end{equation}
\end{notation}

To perform calculations with elements of $F$ we will use the rectangle diagrams of its elements as described in~\cite{cannon_fp:intro_thompson}. Given an element $f\in F$, its \emph{rectangle diagram} is a rectangle with a top that represents the domain of $f$, and a bottom that represents the range of $f$, in which for every point $t$ on the top where $f$ is not differentiable, we construct a line segment from $t$  to $f(t)$ on the bottom. These diagrams allow us to see how sub-intervals of the domain (in between the points of non-differentiability) are mapped to sub-intervals of the range. This is especially useful when composing functions because we can just stack the diagrams on top of each other.

Aside from being finitely generated, $F$ is known to have a balanced finite presentation with two generators and two relators:
\begin{equation}
\label{eqn:finite_presentation}
F=\langle x_0, x_1 \mid [x_1^{-1}x_0, x_0x_1x_0^{-1}]=[x_1^{-1}x_0, x_0^2x_1x_0^{-2}]=1 \rangle,
\end{equation}
where the generators $x_0,x_1$ correspond to the ones defined in~\eqref{eqn:x_0x_1}. Note that this presentation, as well as the presentation~\eqref{eqn:infinite_presentation} is obtained from the presentations given in~\cite{cannon_fp:intro_thompson} by reversing the involved words since we consider the right actions and~\cite{cannon_fp:intro_thompson} deals with left actions.

Besides the above finite presentation, in many situations it is convenient to consider an infinite, but more symmetric presentation for $F$ given in~\eqref{eqn:infinite_presentation}, in which generators $x_0$ and $x_1$ again correspond to the ones defined in~\eqref{eqn:x_0x_1}.  It immediately follows from~\eqref{eqn:infinite_presentation} that
\[x_n=x_0^{n-1}x_1x_0^{-(n-1)}.\]

Each point of the interval $[0,1]$ can be associated with its binary expansion. This allows us to extend naturally the action of $F$ on the sets $X^\ast$ and $X^\omega$ of all finite and infinite words over the alphabet $X=\{0,1\}$ respectively. The latter set is homeomorphic to a Cantor set and it is easy to see that $F$ acts on it by homeomorphisms.  The action of generators $x_0$ and $x_1$ can be defined as follows:

\begin{equation}
\label{eqn_cantor_action}
\begin{array}{ccc}
x_0:\ \left\{\begin{array}{lll}
0w&\mapsto& 00w,\\
10w&\mapsto& 01w,\\
11w&\mapsto& 1w,
\end{array}\right.&\hspace{1.5cm}&x_1:\ \left\{\begin{array}{lll}
0w&\mapsto& 0w,\\
10w&\mapsto& 100w,\\
110w&\mapsto& 101w,\\
111w&\mapsto& 11w,
\end{array}\right.\\
\end{array}
\end{equation}
where $w$ is an arbitrary word in $X^\omega$. These homeomorphisms are rational homeomorphisms of $X^\omega$ that can be defined by finite state asynchronous automata. The group of all rational homeomorphisms of $X^\omega$ was introduced in~\cite{gns00:automata} and was studied extensively by Belk, Bleak, Matucci, and Zaremsky in connection to Gromov hyperbolic groups~\cite{belk_bm:rational_embeddings_of_hyperbolic_groups21} and the Boone-Higman conjecture for this class~\cite{belk_bmz:hyperbolic_boone-higman}.

Note that the dyadic rational numbers in $[0,1]$ correspond to the sequences ending in $0^\infty$ or $1^\infty$, and rational points in $[0,1]$ correspond to eventually periodic sequences $vw^{\infty}$ in $X^\omega$. We will call such sequences \emph{rational} elements of $X^\omega$. All other elements we will call \emph{irrational}. There is a one-to-one correspondence between irrational elements of $X^\omega$ and irrational numbers in $[0,1]$.

In the description of Schreier graphs below we will refer to the one-sided shift on $X^\omega$, denoted by $\sigma$ and defined by
\[\sigma(a_1a_2a_3\ldots)=a_2a_3a_4\ldots\]
for $a_1a_2a_3\ldots\in X^\omega$.

\section{Schreier graphs of stabilizers of rational points}
\label{sec:schreier}

In order to construct explicit generating sets for the stabilizer subgroups in $F$, we will apply the Reidemeister-Schreier procedure to the Schreier graphs of these stabilizers. To explain the construction explicitly, we start from recalling the definition of Schreier graphs and the description of the Schreier graphs of the stabilizers of rational points of $X^\omega$ obtained in~\cite{savchuk:thompson2}.

\begin{definition}
Let $G$ be a group generated by a finite generating set $S$ acting on a set $M$. The \emph{(orbital) Schreier graph} $\Gamma(G,S,M)$ of the action of $G$ on $M$ with respect to the generating set $S$ is an oriented labeled graph defined as follows. The set of vertices of $\Gamma(G,S,M)$ is $M$ and there is an arrow from $a\in M$ to $b\in M$ labeled by $s\in S$ if and only if $a^s=b$. A Schreier graph with a selected vertex is called a \emph{pointed Schreier graph} and the selected vertex is then called the \emph{base point} of this graph.
\end{definition}

An equivalent alternative view on Schreier graphs goes back to Schreier. For any subgroup $H$ of a group $G$, the group $G$ acts on the set of right cosets $G/H$ by right multiplication. The corresponding Schreier graph $\Gamma(G,S,G/H)$ is denoted as $\Gamma(G,S,H)$ or just $\Gamma(G,H)$ if the generating set is clear from the context. The graph $\Gamma(G,S,H)$ is often considered as a pointed graph by selecting the vertex $H$ as a base point.

Conversely, if $G$ acts on $M$ transitively, then $\Gamma(G,S,M)$ is canonically isomorphic to $\Gamma(G,S,\Stg(a))$ for any $a\in M$, where the vertex $b\in M$ in $\Gamma(G,S,M)$ corresponds to the coset from $G/\Stg(a)$ consisting of all elements of $G$ that move $a$ to $b$.  Also, to simplify notation, we will call $\Gamma(G,S,\Stg(a))$ simply the Schreier graph of $a$ and denote by $\Gamma_a$ when the group, the set and the action are clear from context.

The Schreier graphs of the stabilizers of points of the Cantor set $X^\omega$ with respect to the generating set $\{x_0,x_1\}$ are constructed in Theorem~6.1 of~\cite{savchuk:thompson2}. First, the Schreier graph $\Gamma_{10^\infty}$ of the stabilizers of $10^\infty$  (corresponding to $\frac 12\in [0,1]$) shown in Figure~\ref{fig:graph12} was described in~\cite{savchuk:thompson}. In Figure~\ref{fig:graph12} the vertices are labeled by the nonzero prefixes of the corresponding points of $X^\omega$.
\begin{figure}
\hspace{1cm}\epsfig{file=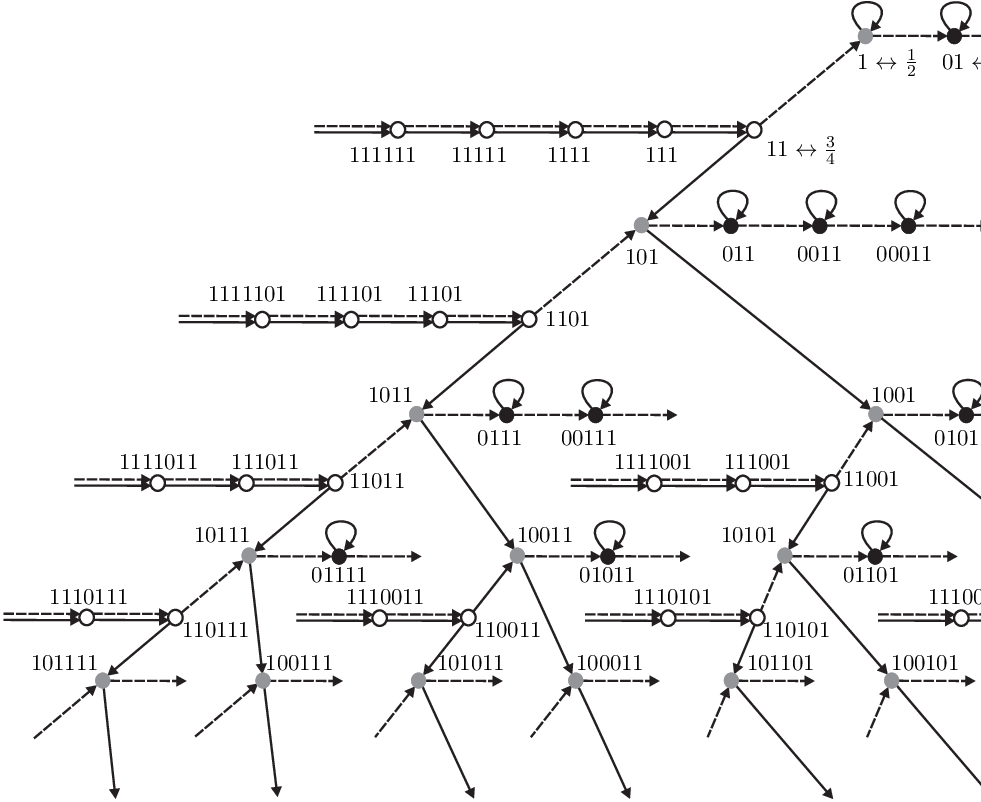,width=300pt}
\caption{The Schreier graph of the stabilizer of $10^\infty$\label{fig:graph12}}
\end{figure}

Since $F$ acts transitively on the set of dyadic rational numbers from $(0,1)$, this graph is also the Schreier graph of any such number, in which the base point is moved to the vertex corresponding to that number. This is also the case for  any point of the Cantor set $X^\omega$ that ends with $0^\infty$ or $1^\infty$ except $0^\infty$ and $1^\infty$ themselves. We comment on this in Remark~\ref{rem:01} in the end of the paper.

For other rational points of $X^\omega$ we include here the description of Schreier graphs from~\cite{savchuk:thompson2} since we will use their structure for the calculations in Section~\ref{sec:generators}:

\begin{theorem}[\mbox{\cite[Theorem 6.1]{savchuk:thompson2}}]
\label{thm_schreier_rational}
Each rational point $b$ of the Cantor set $X^\omega$ except $0^\infty$ and $1^\infty$ can be uniquely written as either $b=1^n0vw^\infty$ or $b=0^m1v^\infty$ for some $v=v(0,1),w=w(0,1)\in X^*$ such that $w$ is not a proper power and the ending of $v$ differs from the one of $w$. The (pointed) Schreier graph $\Gamma_b$ of the action of $F$ on the orbit of $b$ is depicted in Figure~\ref{fig_correspondence_rational} and has the following structure:
\begin{enumerate}
\item the base point is labeled by $b$;
\item each vertex labeled by $10u$ for $u\neq \sigma^{i}(w^\infty)$ is a root of the tree hanging down from this vertex that is canonically isomorphic to the tree hanging down at the vertex $101$ in the Schreier graph $\Gamma_{10^\infty}$ shown in Figure~\ref{fig:graph12};
\item a path going up from vertex $10vw^\infty$ to vertex $10\sigma^{|w|-1}(w^\infty)$ turns left or right at the $k$-th step, if the $k$-th letter in $vw$ is 0 or 1 respectively;
\item if the first letter of $w$ is 0 then there is an edge from the vertex $10w^\infty$ to the vertex $10\sigma^{|w|-1}(w^\infty)$ labeled by $x_1$ (see Figure~\ref{fig_correspondence_rational});
\item if the first letter of $w$ is 1 then the vertex $110w^\infty$ is adjacent to vertices $10w^\infty$ and $10\sigma^{|w|-1}(w^\infty)$ via edges labeled by $x_0$ and $x_1$ respectively;
\item the path in the third item for the empty word $v$ and the edge or the path in the previous two items create a loop in the Schreier graph corresponding to $w$ which we will call the \emph{nontrivial loop};
\item each vertex adjacent to the grey vertex in the nontrivial loop, and not belonging to this loop is either
\begin{itemize}
\item the beginnings of geodesics isomorphic to geodesic $(1,01,001,\ldots)$ in $\Gamma_{10^\infty}$, or
\item the root of the tree hanging down and to the left at the vertex $11$ in $\Gamma_{10^\infty}$, or
\item the root of the tree hanging down and to the right at the vertex $101$ in $\Gamma_{10^\infty}$.
\end{itemize}
\end{enumerate}
\end{theorem}

\begin{figure}
\begin{center}
\epsfig{file=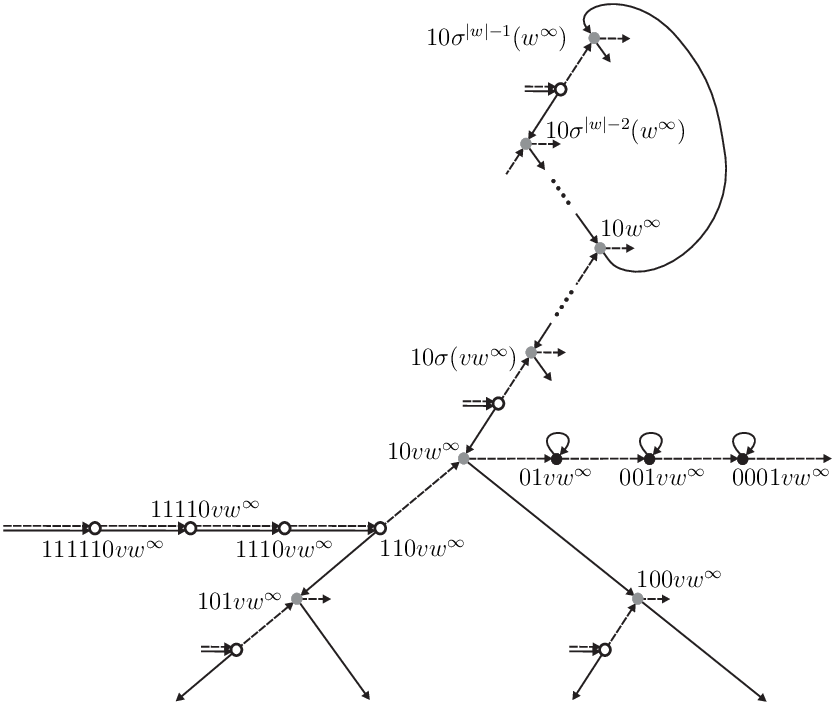,width=300pt}
\end{center}
\caption{The Schreier graph of the stabilizer of a rational point of the Cantor set\label{fig_correspondence_rational}}
\end{figure}

We will modify the description above into the one that works better for our purposes. Namely, the nontrivial loop in Figure~\ref{fig_correspondence_rational} can be shifted down, so that all of its vertices other than $10w^\infty$ itself are below the vertex $10w^\infty$ (which will be on the top of the loop now and will play a role of the root of the tree-like structure made of grey vertices). In this case every grey vertex in $\Gamma_b$ can be identified by a finite word (not necessarily unique) over the alphabet $\{A,B\}$, where $A$ and $B$ correspond to $x_0^{-1}x_1$ and $x_1$: to a word $W(A,B)\in\{A,B\}^*$ we associate the grey vertex that is the endpoint of the path in $\Gamma_b$ initiating at $10w^\infty$ and labeled by $W(x_0^{-1}x_1,x_1)$. In other words, $\Gamma_b$ has now a tree-like structure, in which the root $10w^\infty$ is labeled by the empty word $\varepsilon$ over $\{A,B\}$ and if a grey vertex has a label $W$, then the labels of the grey vertices right below it are $WA$ (the left one) and $WB$ (the right one).

For a finite word $U=a_1a_2\ldots a_n$ (over any alphabet) we denote by $U^R$ the \textit{reverse word} obtained by reading $U$ from right to left, i.e., $U^R=a_na_{n-1}\ldots a_1$. Thus, taking into account the notation given in the introduction, for the word $w\in\{0,1\}^*$ representing the period in $b$, $w^R(B,A)$ denotes a word over $\{A,B\}$ obtained from $w^R$ by replacing $0$ and $1$ by $B$ and $A$, respectively. With the above identification of the set of vertices of $\Gamma_b$ and $\{A,B\}^*$, each grey vertex in $\Gamma_b$, except those labeled by the prefixes of $w^R(B,A)$ (vertices of the nontrivial loop), is the root of the tree hanging down from this vertex that is canonically isomorphic to the tree hanging down at the vertex $101$ in the Schreier graph $\Gamma_{10^\infty}$ shown in Figure~\ref{fig:graph12} (this corresponds to item 2 in Theorem~\ref{thm_schreier_rational}). The vertex labeled by the prefix of $w^R(B,A)$ of length $|w|-1$ has only one grey vertex right below it, while the second connection (either via an edge labeled by $x_1$, if the first letter of $w$ is 0, or via a path of length 2 labeled by $x_0^{-1}x_1$, if the first letter of $w$ is 1) joins it to the root vertex labeled by $10w^\infty$.

To illustrate the construction, we give an example of the Schreier graph of the stabilizer of $10(0100)^\infty$ (that corresponds to the rational number $\frac4{15}\in (0,1)$) in Figure~\ref{fig:schreier0100}. Here we have $w(0,1)=0100$, so $w^R(B,A)=BBAB$ with the prefix $BBA$ of length $|w|-1=3$. Thus, on the figure the vertex $BBA=B^2A$ corresponding to the point $(10)100(0100)^\infty=10\sigma(w^\infty)$ is connected by an edge denoted by $e_w$ and labeled by $x_1$ (since the first letter of $w$ is 0) to the root vertex $10(0100)^\infty$.

\begin{figure}
\epsfig{file=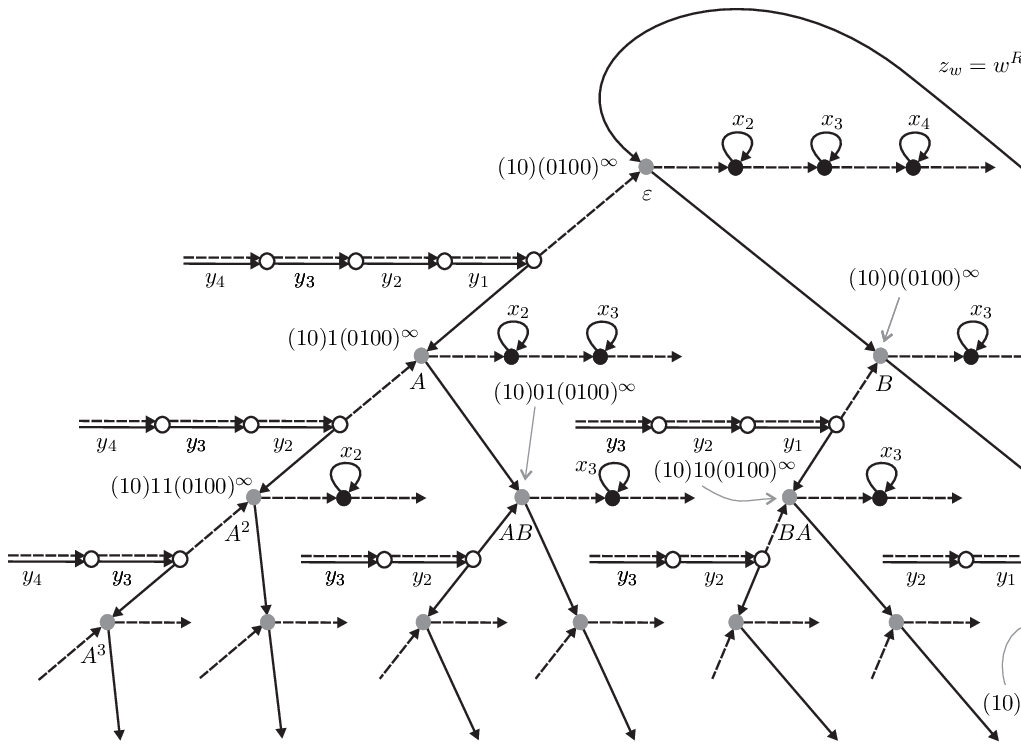,width=320pt}
\vspace{1cm}
\caption{The Schreier graph of the stabilizer of $(0100)^\infty\leftrightarrow\frac4{15}$\label{fig:schreier0100}}
\end{figure}

Now we can see that all the grey vertices in $\Gamma_b$ can be identified \emph{uniquely} by words over $\{A,B\}^*$ that do not have $w^R(B,A)$ as a prefix.

\section{Generating set for the stabilizers of rational points}
\label{sec:generators}

In this section we apply Reidemeister-Schreier procedure to Schreier graphs $\Gamma_b$ to compute the explicit generating sets for the stabilizers of rational points of $X^\omega$. 

The Reidemeister-Schreier procedure works as follows (see~\cite{lyndon_s:comb_grp_theory} for the combinatorial description of the procedure equivalent to the explanation below). Suppose we have a group $G$ generated by a finite set $S=\{s_1, s_2, \ldots , s_n\}$ that has a presentation $G\cong\langle S\mid R\rangle$, where $R$ is a set of relators written as reduced words in $S\cup S^{-1}$. Let $H$ be a subgroup of $G$ with the corresponding pointed Schreier graph $\Gamma=\Gamma(G,S,H)$ with $H$ being the base point in $\Gamma$ (so that $H$ is the stabilizer of this base point). The Reidemeister-Schreier procedure constructs the presentation for $H$ via the following steps.

\begin{enumerate}
\item Choose a spanning tree $T$ (a subtree in the graph that includes all the vertices) in $\Gamma$.
\item Each edge $e$ in $\Gamma$ that is not in $T$ gives rise to a Schreier generator of $H$. Specifically, this generator is given by a word over $\{s_1, \dots , s_n,s_1^{-1},\ldots, s_n^{-1}\}$ that is read along the path of the form $l_{i(e)}\cdot e \cdot l_{t(e)}^{-1}$, where $i(e)$ and $t(e)$ are the initial and terminal vertices of $e$, respectively, and for a vertex $v$ of $\Gamma$ $l_{v}$ is the unique path in the spanning tree from the base point in $\Gamma$ to $v$.
\item The collection $S_H$ of all Schreier generators constructed in part 2 forms a generating set for $H$.
\item For each vertex $v$ of $\Gamma$ and each relator $r\in R$ construct a Schreier relator of $H$ by reading $r$ in $\Gamma$ starting from $v$ and recording the generators from $S_H$ corresponding to the edges not in the spanning tree $T$.
\item The collection $R_H$ of all Schreier relators constructed in part 4 is the set that defines the presentation of $H$ as $H\cong\langle S_H\mid R_H\rangle$.
\end{enumerate}

If $G$ is finitely presented and $H$ has a finite index in $G$ the procedure immediately yields a finite presentation for $H$. However, it can also be applied to subgroups of infinite index and, as we will see below, can still yield a finite presentation.

We will now apply the Reidemeister-Schreier procedure to $\St_F(a)$ for a rational point $a=10w^\infty$ of the Cantor set $X^\omega$ with a finite word $w$ that is not a proper power, using the description of the corresponding pointed Schreier graph $\Gamma_a$ that we discussed in the previous section, in which point $a$ plays a role of the root of the tree consisting of grey vertices. There is an almost natural choice of a spanning tree $T$ in $\Gamma_a$ consisting of all edges except the loops at the black rays, the bottom edges of the white rays, and the edge, that we denote by $e_w$, labeled by $x_1$ coming to the root vertex $10w^\infty$ either from a grey vertex $10\sigma(w)^\infty$ if the first letter of $w$ is 0 (as shown in Figure~\ref{fig:schreier0100}) or from a white vertex connected to that grey vertex if the first letter of $w$ is 1. By Reidemeister-Schreier procedure $\St_F(a)$ is generated by the Schreier generators corresponding to the edges that are not included in the spanning tree. Let us compute these generators and set up notation for them.

Every edge $e$ that is not in the spanning tree, except for $e_w$, can be identified uniquely by the closest grey vertex to the terminal vertex of $e$, and the distance $n$ from the terminal vertex of $e$ to that grey vertex. More precisely, for each grey vertex $v_W$ of $\Gamma_a$ labeled by $W=W(A,B)$, where $A$ and $B$ correspond to $x_0^{-1}x_1$ and $x_1$, we define two families of Schreier generators of $\St_F(a)$. Let $x_{W,n}$ be the Schreier generator corresponding to the $n$-th loop (counting from the left) along the black ray stemming out of $v_W$, and let $y_{W,n}$ be the Schreier generator corresponding to the $n$-th horizontal bottom edge (counting from the right) in the white ray stemming out of the white vertex right below $v_W$. Finally, let $z_w$ denote the Schreier generator corresponding to the edge $e_w$. From the structure of the graph, according to the Reidemeister-Schreier procedure, we calculate:
\[\begin{array}{lcl}
x_{W,n}&=&W(x_0^{-1}x_1, x_1) \cdot x_0^{n}x_1x_0^{-n} \cdot  W(x_0^{-1}x_1, x_1)^{-1},\quad n\geq 1,\\[2mm]
y_{W,n}&=&W(x_0^{-1}x_1, x_1) \cdot  x_0^{-1}\cdot x_0^{-n}x_{1}x_0^{n-1}x_0 \cdot  W(x_0^{-1}x_1, x_1)^{-1},\quad n\geq 1,\\[2mm]
z_w&=&w^R(x_1,x_0^{-1}x_1).
\end{array}
\]

Thus, according to Reidemeister-Schreier procedure we obtain the following lemma.

\begin{lemma}
\label{lem:infinite_gen_set}
For a rational point $a=10w^\infty$ of $X^\omega$, where $w\in X^*$ is not a proper power, the stabilizer of $a$ in $F$ is generated by the following elements of $F$:
\[St_F(a)=  \langle x_{W,n}, y_{W,n}, z_w \mid n \geq 1, W\in\{A,B\}^*\ \text{that does not have $w^R(B,A)$ as a prefix}\rangle.\]
\end{lemma}

Most of the arguments below are devoted to the simplification of the generating set in the above lemma. First, we observe that
\[x_{\varepsilon,n}=x_0^{n}x_1x_0^{-n}=x_{n+1},\]
where $x_{n+1}$ is one of the generators of $F$ from the infinite presentation~\eqref{eqn:infinite_presentation} of $F$. We also denote
\[y_n:=y_{\varepsilon,n}=x_0^{-1}\cdot x_0^{-n}x_{1}x_0^{n-1}x_0=x_0^{-n-1}x_{1}x_0^{n}.\]

It is shown in~\cite[Example~2.3]{cannon_fp:intro_thompson} that the action of $x_n$ on $[0,1]$ is as follows
\[\begin{array}{llll}
x_n(t) & =&
\begin{cases}
t, & 0\leq t \leq 1-\frac1{{2}^{n}},\\
\frac{t}2+\frac12-\frac1{2^{n+1}}, & 1-\frac1{2^n} < t \leq 1-\frac1{2^{n+1}},\\
t-\frac1{2^{n+2}}, & 1-\frac1{2^{n+1}} < t \leq 1-\frac1{2^{n+2}},\\
2t-1, &1-\frac1{2^{n+2}} < t \leq1,\\
\end{cases}
\end{array}
\]
or, equivalently, $x_n$ can be represented by the rectangle diagram below.

\vspace{3mm}
\hspace{4.1cm}
\beginpicture
\setcoordinatesystem units <4mm, 4mm> point at 0 0
\setlinear
 \plot 8 2          20 2   20 4   8 4   8 2 /
 \plot 17 2        17 4 /
 \plot 18.5 4     17.75 2 /
 \plot 19.25 4   18.5 2 /
\put {$1-\frac1{2^{n}}$}  at 17 5
\put {$x_n=$}   at 6 3
\endpicture
\vspace{0.8cm}

In the next lemma we will compute the action of $y_n$ on $[0,1]$.

\begin{lemma}
\label{lem:y}
The function $y_n$ is given by:
\[\begin{array}{llll}
y_n(t) & =&
\begin{cases}
2t, & 0\leq t \leq \frac1{2^{n+2}}\\
t+\frac1{2^{n+2}}, & \frac1{2^{n+2}}<t \leq \frac1{2^{n+1}},\\
\frac{t}2+\frac12-\frac1{2^{n+1}}, & \frac1{2^{n+1}} < t \leq \frac1{2^n},\\
t, &\frac1{2^n} < t \leq1,\\
\end{cases}
\end{array}
\]
or, equivalently, by the following rectangle diagram:

\vspace{3mm}
\hspace{4.1cm}
\beginpicture
\setcoordinatesystem units <4mm, 4mm> point at 0 0
\setlinear
 \plot 8 2          20 2   20 4   8 4   8 2 /
 \plot 11 4   11 2 /
 \plot 9.5 4   10.25 2 /
 \plot 8.75 4   9.5 2 /

\put {$\frac1{2^{n}}$}  at 11 5
\put {$y_n=$}   at 6 3
\endpicture
\vspace{0.8cm}
\end{lemma}

\begin{proof}
We prove the claim by induction on $n$. For the base case $n=1$ we have $y_1=x_0^{-2}x_1x_0$, so we calculate the expression for $y_1$ by rectangle diagrams (note that since here, unlike in~\cite{cannon_fp:intro_thompson}, we consider right actions, we juxtapose the rectangle diagrams from top to bottom).

\vspace{3mm}
\hspace{2cm}
\beginpicture
\setcoordinatesystem units <4mm, 4mm> point at 0 0
\setlinear
\plot 2 0   14 0   14 8   2 8   2 0 /
\plot 2 6   14 6 /
\plot 2 4   14 4 /
\plot 2 2   14 2 /
\plot 18 2   30 2   30 4   18 4   18 2 /
\plot 8 4   8 2 /

\plot 11 4   9.5 2 /
\plot 12.5 4   11 2 /

\plot 8 8   11 6 /
\plot 5 8   8 6 /
\plot 8 6   11 4 /
\plot 5 6   8 4 /
\plot 8 2   5 0 /
\plot 11 2   8 0 /

 \plot24 4   24 2 /
 \plot 21 4   22.5 2 /
 \plot 19.5 4   21 2 /

\setdashes<1mm>

\plot 11 6   12.5 4 /

\plot 3.5 8   5 6 /
\plot 9.5 2   6.5 0 /

\put {$x_0^{-1}$=}   at 0.7 7
\put {$x_0^{-1}$=}   at 0.7 5
\put {$x_1$=}   at 1 3
\put {$x_0$=} at 1 1
\put {=} at 16 3
\put {= $y_{1}$} at 31.9 3
\put {$\frac1{2}$} at 8 9
\put {$\frac1{2}$} at 24 5
\put {$\frac1{4}$} at 5 9
\put {$\frac1{4}$} at 21 5
\put {$\frac1{8}$} at 19.5 5
\endpicture
\vspace{0.5cm}

The induction step is proved similarly. Assuming that the claim holds true for $y_n$ for some fixed $n$, since $y_{n+1}=x_0^{-1}y_nx_0$, we calculate using the corresponding rectangle diagrams, where we concentrate only on the left side because $x_0$ coincides with $\frac t2$ and $x_0^{-1}$ coincides with $2t$ for all $t\in[0,\frac14]$.

\vspace{3mm}
\hspace{1cm}
\beginpicture
\setcoordinatesystem units <4mm, 4mm> point at 0 0
\setlinear

\plot 2 0   14 0 /
\plot 2 2   14 2 /
\plot 2 6   14 6 /
\plot 2 4   14 4 /
\plot 2 2   14 2 /
\plot 2 0   2 6 /
\plot 30 4   18 4   18 2   30 2 /
\plot 5 4   6.5 2 /
\plot 3.5 4   5 2 /
 \plot21 4   21 2 /
 \plot 19.5 4   20.25 2 /
 \plot 18.75 4   19.5 2 /
\plot 8 4   8 2 /

\setdashes<1mm>

\plot 3.5 6   5 4 /
\plot 2.75 6   3.5 4 /
\plot 5 2   3.5 0 /
\plot 6.5 2  4.25 0 /
\plot 5 6   8 4 /
\plot 8 2   5 0 /

\put {$x_0^{-1}$=}   at 0.7 5
\put {$y_n$=}   at 1 3
\put {$x_0$=} at 1 1
\put {=} at 16 3
\put {= $y_{n+1}$} at 31.9 3
\put {$\frac1{2^n}$} at 8 7
\put {$\frac1{2^{n+1}}$} at 21 5
\put {$\frac1{2^{n+1}}$} at 5 7
\endpicture

\end{proof}

\begin{corollary}
\label{cor:xy_commute}
For each $n\geq 1, m\geq 1$ the elements $x_n$ and $y_m$ commute.
\end{corollary}

\begin{proof}
The claim follows from the fact that the supports of $x_n$ and $y_m$ are disjoint: $x_n$ moves points only in $1X^\omega$ and $y_m$ only in $0X^\omega$.
\end{proof}

To compute the relations between $y_i$'s we consider the following transformation of $F$ (to be justified in the next lemma):
\[
\begin{array}{crcl}
\varphi\colon&F&\longrightarrow&F\cr
&f(t)&\longmapsto&1-f(1-t)
\end{array}
\]
where $t$ is a real number.

\begin{lemma}
\label{lem:symmetry}
For each $f\in F$,  $\varphi(f)$ is an element of $F$ whose graph is the image under the central symmetry (with respect to the center $(\frac12,\frac12)$ of $[0,1]^2$) of the graph of $f$.
\end{lemma}
\begin{proof}
The point $(t,f(t))$ is on the graph of $f$ if and only if $(1-t,1-f(t))$ is a point on the graph of $\varphi(f)(1-t)$ because
\[\big(1-t,1-f(t)\big)=\big(1-t, 1-f(1-(1-t))\big)=\big(1-t, \varphi(f)(1-t)\big).\]
Therefore, $\varphi(f)$ is also an increasing piecewise linear homeomorphisms from $[0,1]$ to itself that is differentiable everywhere except at finitely many dyadic rational numbers and such that on the intervals of differentiability the derivatives are integer powers of 2. Thus, $\varphi(f)$ is an element of $F$.
\end{proof}

\begin{lemma}
\label{lem:automorphism}
The map $\varphi$ is an automorphism of $F$ of order 2.
\end{lemma}

\begin{proof}
Since by Lemma~\ref{lem:symmetry} the graph of $\varphi(f)$ is symmetric to the graph of $f$, we immediately obtain that $\varphi$ is an involution, so it is both one-to-one and onto.
We only need to check that $\varphi$ is a homomorphism.
\begin{equation}
\begin{split}
\varphi(f(t)\circ g(t))& = \varphi(g(f(t))\\& = 1-g(f(1-t)) = 1-g(1-1+f(1-t))\\& = 1-g(1-[1-f(1-t)])\\&= [1-f(1-t)] \circ [1-g(1-t)]\\&=\varphi(f(t)) \circ \varphi(g(t))
\end{split}
\end{equation}
Thus, $\varphi$ is an automorphism.
\end{proof}

\begin{lemma}
\label{lem:phi_xy}
For each $n\geq 1$
\[\varphi(x_n)=y_n.\]
\end{lemma}

\begin{proof}
If an element $f\in F$ sends a point $a\in[0,1]$ to a point $b=f(a)\in [0,1]$, then
\[\varphi(f)(1-a)=1-f(1-(1-a))=1-f(a)=1-b.\]
Therefore the rectangle diagram of $\varphi(f)$ is obtained from the rectangle diagram of $f$ by applying a reflection about the vertical symmetry axis of the rectangle.

\hspace{2.8cm}\beginpicture
\setcoordinatesystem units <4mm, 4mm> point at  0 0
\setlinear
\plot 0 0     12 0   12 2   0 2   0 0 /
\plot 3 2      9 0 /
\plot 18 0    30 0   30 2   18 2   18 0 /
\plot 21  0   27  2 /
\plot 13 1    17 1 /
\plot 16.8    1.2   17 1  /
\plot 16.8    0.8  17 1  /
\put    {$a$}          [b] at 3 2.2
\put    {$b$}          [b] at 9 -1
\put    {$1-a$}       [b] at 27 2.2
\put    {$1-b$}       [b] at 21 -1
\put    {$\varphi$} [b] at 15 1.2
\endpicture
\vspace{2mm}

Thus, the statement of the lemma immediately follows from Lemma~\ref{lem:y} and the rectangle diagram for $x_n$.
\end{proof}

\begin{corollary}
\label{cor:y_rel}
For each $n\geq 1$ and $0\leq k < n$
\[y_ky_ny_k^{-1}=y_{n+1}.\]
\end{corollary}

\begin{proof}
Follows from Lemma~\ref{lem:phi_xy} and the relations from the presentation~\eqref{eqn:infinite_presentation}.
\end{proof}

We are now ready to simplify the generating set for $\St_F(a)$ from Lemma~\ref{lem:infinite_gen_set}. For any word $W\in \{A,B\}^*$ let $W_A$ and $W_B$ denote the number of $A$'s and $B$'s in $W$, respectively. We refer the reader to Figure~\ref{fig:schreier0100} to illustrate the proposition below, where edges outside of the spanning tree $T$ are labeled by Schreier generators $x_n, n\geq 2$,  $y_m, m\geq 1$, and $z_w$.

\begin{proposition}
\label{prop:x}
For any integer $n\geq 1$ and any $W\in\{A,B\}^*$ that does not have $w^R(B,A)$ as a prefix we have
\[\begin{array}{l}
x_{W,n}=x_{n+1+W_B},\\
y_{W,n}=y_{n+W_A}.\\
\end{array}\]
\end{proposition}

\begin{proof}
It follows from the presentation~\eqref{eqn:infinite_presentation} that $x_0^nx_1x_0^{-n}=x_{n+1}$ for all $n\geq 1$. Therefore
\[x_{W,n}=W(x_0^{-1}x_1,x_1) \cdot x_0^nx_1x_0^{-n}\cdot W(x_0^{-1}x_1,x_1)^{-1}=W(x_0^{-1}x_1,x_1) \cdot x_{n+1}\cdot W(x_0^{-1}x_1,x_1)^{-1}.\]

Since for any $i\geq 2$
\[(x_0^{-1}x_1)x_i(x_0^{-1}x_1)^{-1}=x_0^{-1}\cdot (x_1x_ix_1^{-1})\cdot x_0^{-1}=x_0^{-1}\cdot (x_{i+1})\cdot x_0=x_i\]
and
\[x_1x_ix_1^{-1}=x_{i+1},\]
the conjugation of $x_{n+1}$ by $W(x_0^{-1}x_1,x_1)$ increases the index exactly by $W_B$ and $x_{W,n}=x_{n+1+W_B}$.

Similarly,
\[y_{W,n}=W(x_0^{-1}x_1,x_1)\cdot x_0^{-1}\cdot x_0^{-n}x_1x_0^n\cdot W(x_0^{-1}x_1,x_1)^{-1}=W(x_0^{-1}x_1,x_1)\cdot y_{n}\cdot W(x_0^{-1}x_1,x_1)^{-1}.\]
It follows from Corollary~\ref{cor:xy_commute} and the definition of $y_n$ that for $i\geq 1$
\[(x_0^{-1}x_1)y_i(x_1^{-1}x_0)=x_0^{-1}y_i x_0=y_{i+1}\]
and
\[x_1y_ix_1^{-1}=y_{i}.\]
Therefore, the conjugation of $y_{n}$ by $W(x_0^{-1}x_1,x_1)$ increases the index exactly by $W_A$ and $y_{W,n}=y_{n+W_A}$.
\end{proof}

Finally we prove the main result of the paper, Theorem~\ref{thm:main} from Introduction, that we repeat here for convenience of the reader.

\begin{theoremmain}~\\
\begin{enumerate}
  \item[(a)] For a rational point $a=10w^\infty$ of the Cantor set $X^\omega$ with a finite word $w$ that is not a proper power, we have
  \[\St_F(a)=\langle x_2, x_3, y_1, y_2, w(x_1^{-1},x_1^{-1}x_0)\rangle.\]
  \item[(b)] Let $b\in X^\omega$ be an arbitrary rational different from $0^\infty$ and $1^\infty$ point and let $b=vw^\infty$ be its unique decomposition as an eventually periodic word with a finite nonempty period $w\in X^*$ that is not a proper power and a finite preperiod $v\in X^*$ such that the ending of $v$ differs from the one of $w$. Then there is $h\in F$ that can be explicitly computed from the Schreier graph $\Gamma_b$, such that $h(b)=10vw^\infty$ and
  \[\St_F(b)=\langle x_2^h, x_3^h, y_1^h, y_2^h, w(x_1^{-1},x_1^{-1}x_0)^h\rangle.\]
\end{enumerate}
\end{theoremmain}

\begin{proof}
For part (a) we start from a generating set
\[\left\{x_{W,n}, y_{W,n}, z_w \mid n \geq 1, W\in\{A,B\}^*\ \text{that does not have $w^R(B,A)$ as a prefix}\right\}\]
constructed in Lemma~\ref{lem:infinite_gen_set}. By Proposition~\ref{prop:x} all generators $x_{n,W}$ and $y_{n,W}$ are equal to some of the $x_i$, $i\geq 2$ and $y_j$, $j\geq 1$. On the other hand, $x_i=x_{\varepsilon,i-1}$ for $i\geq 2$ and $y_j=y_{\varepsilon,j}$ for $j\geq 1$. Therefore, we can conclude that
\[St_F(a)=  \langle x_i,y_j,z_w\colon i\geq 2, j\geq 1\rangle.\]
Now, since $x_n=x_2^{n-3}x_3x_2^{-(n-3)}$ for $n\geq 4$ by presentation~\eqref{eqn:infinite_presentation} and $y_n=y_1^{n-2}y_2y_1^{-(n-2)}$ for $n\geq 3$ by Corollary~\ref{cor:y_rel}, we reduce the generating set to
\[St_F(a)=  \langle x_2,x_3,y_1,y_2,z_w\rangle.\]
Finally, the replacement of $z_w$ by $z_w^{-1}$:
\[z_w^{-1}=w^R(x_1,x_0^{-1}x_1)^{-1}=w(x_1^{-1},(x_0^{-1}x_1)^{-1})=w(x_1^{-1},x_1^{-1}x_0)\]
finishes the proof.

Part (b) follows from the transitivity of the action of $F$ on the set of rational points with the same period, as follows directly from Theorem~\ref{thm_schreier_rational} and was also proved by Belk and Matucci (see~Proposition 3.2.3 in~\cite{matucci:phd08}). Indeed, given an element $h\in F$ such that $h(b)=10w^\infty$, we get $\St_F(b)=h\St_F(10w^\infty)h^{-1}$, which proves the claim. The element $h$ can be calculated as a word over $\{x_0,x_1,x_0^{-1},x_1^{-1}\}$ that is read along a path in $\Gamma_b$ connecting $b$ to $10w^\infty$. This word can be clearly effectively computed given the preperiod $v$, but we leave writing out the explicit algorithm that involves a few cases to the reader.
\end{proof}

In the last proposition we show that the Reidemeister-Schreier procedure yields finite presentability of stabilizers of all rational points in $X^\omega$, which, as stated in the introduction, also follows from Lemma~4.11 in~\cite{golan_s:stabilizers_of_finite_sets17}. We will not write out the exact presentation, but it certainly can be done if needed in each individual case.

\begin{proposition}
\label{prop:fp}
For each rational point $b$ of $X^\omega$ the stabilizer $\St_F(b)$ is finitely presented.
\end{proposition}

\begin{proof}
The claim is obvious for $b=1^\infty$ and $b=0^\infty$ as in these cases $\St_F(b)=F$, which is finitely presented. By conjugation, if necessary, we can assume that $b=10w^\infty$ and thus by Theorem~\ref{thm:main} $\St_F(b)=\langle x_2,x_3,y_1,y_2,w(x_1^{-1},x_1^{-1}x_0)\rangle$. We will show that it is enough to take only a finite number of relators from the set of Schreier relators $R_{\St_F(b)}$ constructed in Step~4 of the Reidemeister-Schreier procedure. Indeed, since the length of the longest relator in~\eqref{eqn:finite_presentation} is 14, any Schreier relator from $R_{\St_F(b)}$ that is obtained by reading a relator from $R=\{[x_1^{-1}x_0, x_0x_1x_0^{-1}],[x_1^{-1}x_0, x_0^2x_1x_0^{-2}]\}$ starting from a vertex in $\Gamma_b$ with distance at least 7 from the edge $e_w$ will involve only generators of $\St_F(b)$ of the form $x_i$ or $y_j$ for some $i\geq 2$, $j\geq 1$. It follows from presentation~\eqref{eqn:infinite_presentation} that
\[\langle x_2,x_3,x_4,\ldots\rangle=\langle x_2,x_3\rangle\cong F\]
and from Lemma~\ref{lem:phi_xy} that
\[\langle y_1,y_2,y_3,\ldots\rangle=\langle y_1,y_2\rangle\cong\langle x_1,x_2\rangle\cong F.\]
Thus, all of such relators will follow from 2 relators between $x_2$ and $x_3$, 2 relators between $y_1$ and $y_2$ (coming from a finite presentation~\eqref{eqn:finite_presentation} of $F$), and 4 commuting relators $[x_i,y_j]$, $i=2,3$, $j=1,2$. Note that the commuting relations hold true by Corollary~\ref{cor:xy_commute} and are consequences of a finite number of Schreier relations.

Since there is only finite number of vertices in $\Gamma_b$ which are distance at most 6 from the edge $e_w$, and each such vertex will introduce 2 Schreier relators from the Reidemeister-Schreier procedure, the claim follows.
\end{proof}

We finish the paper with a couple of remarks and a question.

\begin{remark}
In the case of $\St_F(10^\infty)=\St_F(\frac12)$ the nontrivial loop in the Schreier graph $\Gamma_{10^\infty}$ becomes a degenerate loop consisting just of one edge labeled by $x_1$ from the vertex $10^\infty$ to itself, i.e., we obtain the Schreier graph shown in Figure~\ref{fig:graph12}. In this case we have $w=0\in X$ and $z_{w}=z_0=x_1$, so Theorem~\ref{thm:main} yields $\St_F(10^\infty)=\langle x_2, x_3, y_1, y_2, x_1\rangle$.
Since $x_3=x_1x_2x_1^{-1}$ the last equality simplifies to
\[\St_F(10^\infty)=\langle x_1, x_2, y_1, y_2\rangle\cong\langle x_1,x_2\rangle \times \langle y_1,y_2\rangle\cong F\times F,\]
which is the well known generating set and structure description for $\St_F(10^\infty)$.
\end{remark}

\begin{remark}
\label{rem:01}
While in the interval $(0,1)$ we have $0.v10^\infty=0.v01^\infty$, in the Cantor set $X^\omega$ $v10^\infty$ and $v01^\infty$ are 2 distinct points and they have disjoint orbits under the action of $F$. Since the action of $F$ on $X^\omega$ is induced by the action on $[0,1]$, the corresponding Schreier graphs $\Gamma_{v10^\infty}$ and $\Gamma_{v10^\infty}$ must be isomorphic and the stabilizers in $F$ of these points must be equal. This can be easily verified by applying Theorems~\ref{thm_schreier_rational} and~\ref{thm:main} in this case.
\end{remark}

Since by Theorem~5.4 of~\cite{golan_s:stabilizers_of_finite_sets17} the stabilizer of any non-dyadic rational number can be generated by 3 elements, it is natural to ask:

\begin{question}
Are the generating sets of $\St_F(b)$ produced in Theorem~\ref{thm:main} minimal for non-dyadic rational points $b$?
\end{question}

%\bibliographystyle{plain}
%\bibliography{../../mylib}

\def\cprime{$'$} \def\cydot{\leavevmode\raise.4ex\hbox{.}} \def\cprime{$'$}
  \def\cprime{$'$} \def\cprime{$'$} \def\cprime{$'$} \def\cprime{$'$}
  \def\cprime{$'$} \def\cprime{$'$} \def\cprime{$'$} \def\cprime{$'$}
  \def\cprime{$'$} \def\cprime{$'$} \def\cprime{$'$} \def\cprime{$'$}
  \def\cprime{$'$}

\end{document}